\documentclass{article}
\usepackage[utf8]{inputenc}
\usepackage{amsmath}
\usepackage[utf8]{inputenc}
\usepackage{t1enc}
\usepackage{geometry}
\usepackage{amssymb}
\usepackage{amsthm}
\usepackage{amsmath}
\usepackage{verbatim}
\usepackage{graphicx}

\usepackage{xr}
\usepackage{subcaption}
\usepackage{xspace} 

\usepackage{mathrsfs}
\usepackage{caption}
\usepackage{float}
\usepackage{pgf,tikz}
\usetikzlibrary{arrows,shapes}
\usetikzlibrary{decorations.text}
\definecolor{mygray}{RGB}{169,169,169}

\usepackage{thmtools}
\usepackage{thm-restate}

\usepackage{hyperref}

\usepackage{cleveref}

\declaretheorem[name=Theorem,numberwithin=section]{thm}

\usepackage{wrapfig}
\usepackage{hyperref}
\usepackage{mathtools}
\usepackage{textcomp}
\usepackage{pdfpages}
\usepackage{pgfplots}
\pgfplotsset{compat=1.15}
\newtheorem{all}[thm]{Theorem}
\newtheorem{lemma}[thm]{Lemma}
\newtheorem{kov}[thm]{Corollary}
\newtheorem{sejt}[thm]{Conjecture}
\newtheorem{kerd}[thm]{Question}
\theoremstyle{definition}
\newtheorem{pl}[thm]{Example}
\newtheorem{mj}[thm]{Remark}
\newtheorem{mydef}[thm]{Definition}

\newtheorem{obs}[thm]{Observation}

\title{Caching and Accumulation Games}
\author{\'Aron J\'anosik$^{1}$ \and Csenge Mikl\'os$^{1}$ \and D\'aniel G. Simon$^{1,2,3}$ \and Krist\'of Z\'olomy$^{1}$}
\date{}
\begin{document}

\maketitle
\begin{abstract}
        We investigate a discrete search game called the \emph{Multiple Caching Game} where the searcher's aim is to find all of a set of $d$ treasures hidden in $n$ locations. 
        Allowed queries are sets of locations of size $k$, and the searcher wins if in all $d$ queries, at least one treasure is hidden in one of the $k$ picked locations.
        Pálvölgyi showed that the value of the game is at most $\frac{k^d}{\binom{n+d-1}{d}}$, with equality for large enough $n$. We conjecture the exact cases of equality. We also investigate variants of the game and show an example where their values are different, answering a question of Pálvölgyi.

        This game is closely related to a continuous variant, \emph{Alpern's Caching Game}, based on which we define other continous variants of the multiple caching game and examine their values.
\end{abstract}

\footnotetext[1]{Institute of Mathematics, E\"otv\"os Lor\'and University, Hungary}
\footnotetext[2]{R\'enyi Institute, Hungary}
\footnotetext[3]{Research supported by ERC Advanced Grant ”GeoScape”}
\section{Introduction}

The Multiple Caching Game, introduced by Pálvölgyi \cite{Domotor} based on Csóka's method in \cite{Csoka}, is defined as follows:
The first player, called the hider (he), initially distributes $d$ treasures in $n$ locations (boxes), he is allowed to hide multiple treasures in one location. Then the second player, called the searcher (she) queries at most $k$ boxes. If at least one of them contains a treasure, the hider reveals and removes one of them, otherwise she loses. The searcher wins if she finds all $d$ treasures. For an example, see figure \ref{multiplegameabra}.

\begin{figure}[H]
\begin{center}
    
\begin{tikzpicture}[scale=0.33]
  \foreach \i in {0,1,2}{
  \draw (0+10*\i,0) rectangle (1+10*\i,1);
  \draw (1.5+10*\i,0) rectangle (2.5+10*\i,1);
  \draw (3+10*\i,0) rectangle (4+10*\i,1);
  \draw (4.5+10*\i,0) rectangle (5.5+10*\i,1);
    }
  \fill[blue] (0.25,0.5) circle (0.15);
  \fill[blue] (0.75,0.5) circle (0.15);

  \fill[blue] (3.5,0.5) circle (0.15);

  \draw[green, ultra thick, ->] (0.5,2) -- (0.5,1.2);
  \draw[green, ultra thick, ->] (3.5,2) -- (3.5,1.2);

  \fill[blue] (10.5,0.5) circle (0.15);

  \fill[blue] (13.5,0.5) circle (0.15);

  \draw[green, ultra thick, ->] (10.5,2) -- (10.5,1.2);
  \draw[green, ultra thick, ->] (15,2) -- (15,1.2);
  
  \fill[blue] (23.5,0.5) circle (0.15);
  
  \draw[green, ultra thick, ->] (20.5,2) -- (20.5,1.2);
  \draw[green, ultra thick, ->] (25,2) -- (25,1.2);

  \draw[->] (7,0.5) -- (9,0.5);
  \draw[->] (17,0.5) -- (19,0.5);
\end{tikzpicture}\caption{An example of the multiple caching game with $n=4$, $d=3$, $k=2$. The hider placed two treasures in box 1, and one in box 3. As her first move, the searcher queried boxes $1$ and $3$, and the hider decided to reveal and remove a treasure from box $1$. In the second move, she queried boxes $1$ and $4$, and he removed a treasure from box $1$. She then repeated this query, and lost the game, since there were no more treasures in either box $1$ or box $4$.}\label{multiplegameabra} 
\end{center}
\end{figure}

Alpern's Caching Game, introduced in a paper of Alpern, Fokkink,  Lidbetter and Clayton \cite{eredeti} is a continuous variant of this game. In this game, the locations are holes of depth $1$, and the hider places the treasures in some positive depth, such that the combined depths of the lowest treasure in each hole is at most $1$.
The searcher is allowed to dig in any hole continuously, revealing a treasure if she reached its depth. The searcher wins if she will have found all treasures after digging a total depth of $k$. For an example, see figure  \ref{Alperngameabra}.

\begin{figure}
  \centering
  \begin{tikzpicture}[scale=0.2]
    \draw[draw=black, fill=yellow] (0,0) rectangle (2,10);
    \draw[draw=white, fill=white] (0,9.99) rectangle (2,10);
    \fill[blue] (1,6.6666667) circle (0.3);

    \draw[draw=black, fill=yellow] (5,0) rectangle (7,10);
    \draw[draw=white, fill=white] (5,10) rectangle (7,10);

    \draw[draw=black, fill=yellow] (10,0) rectangle (12,10);
    \draw[draw=white, fill=white] (10,10) rectangle (12,10);
    \fill[blue] (6,5) circle (0.3);
    \fill[blue] (6,3.333333) circle (0.3);

    \draw[draw=black, fill=yellow] (20,0) rectangle (22,10);
    \draw[draw=white, fill=white] (20,10) rectangle (22,10);

    \draw[draw=black, fill=yellow] (25,0) rectangle (27,10);
    \draw[draw=white, fill=white] (25,10) rectangle (27,10);

    \draw[draw=black, fill=yellow] (30,0) rectangle (32,10);
    \draw[draw=white, fill=white] (30,10) rectangle (32,10);

    \fill[blue] (26,5) circle (0.3);
    \fill[blue] (26,3.333333) circle (0.3);

    \draw[draw=none, fill=white] (20,6.66667) rectangle (22,10);

    \draw[draw=none, fill=white] (25,6.666667) rectangle (27,10);

    \draw[draw=none, fill=white] (30,6.666667) rectangle (32,10);

    \draw[draw=black, fill=yellow] (40,0) rectangle (42,10);
    \draw[draw=white, fill=white] (40,10) rectangle (42,10);

    \draw[draw=black, fill=yellow] (45,0) rectangle (47,10);
    \draw[draw=white, fill=white] (45,10) rectangle (47,10);

    \draw[draw=black, fill=yellow] (50,0) rectangle (52,10);
    \draw[draw=white, fill=white] (50,10) rectangle (52,10);

    \fill[blue] (46,3.333333) circle (0.3);

    \draw[draw=none, fill=white] (40,6.66667) rectangle (42,10);

    \draw[draw=none, fill=white] (45, 5) rectangle (47,10);

    \draw[draw=none, fill=white] (50,5) rectangle (52,10);

        \draw[draw=black, fill=yellow] (60,0) rectangle (62,10);
    \draw[draw=white, fill=white] (60,10) rectangle (62,10);

    \draw[draw=black, fill=yellow] (65,0) rectangle (67,10);
    \draw[draw=white, fill=white] (65,10) rectangle (67,10);

    \draw[draw=black, fill=yellow] (70,0) rectangle (72,10);
    \draw[draw=white, fill=white] (70,10) rectangle (72,10);

    \draw[draw=none, fill=white] (60,5) rectangle (62,10);

    \draw[draw=none, fill=white] (65, 3.3333333) rectangle (67,10);

    \draw[draw=none, fill=white] (70,5) rectangle (72,10);

    \draw[->,thick] (14.5,5) -- (17.5,5);
    \draw[->,thick] (34.5,5) -- (37.5,5);
    \draw[->,thick] (54.5,5) -- (57.5,5);
  \end{tikzpicture}
  \caption{An example of Alpern's caching game with $n=3, d=3, k=2$. The hider placed one treasure in hole $1$ in depth $\frac{1}{3}$, and two in hole $2$ in depths $\frac{1}{2}$ and $\frac{2}{3}$. The searcher starts digging in all 3 holes simultaneously until she finds a treasure in hole $1$. Afterwards, she continues digging in holes 2 and 3 until she finds a treasure in hole $2$. Then she continues digging in holes $1$ and $2$ simultaneously and wins after digging a total of $\frac{1}{2}+ \frac{2}{3}+ \frac{1}{2} < 2$ depth.}
  \label{Alperngameabra}
\end{figure}

This problem first arose in a paper by Alpern, Fokkink, Lidbetter, and Clayton  \cite{eredeti}. The following papers were also published on the game: \cite{Alpern,Csoka,Csoka2,Domotor}.

The value of these games is the probability of the searcher winning if both parties play optimally, we denote this number by $v_M(n,d,k)$ and $v_A(n,d,k)$ respectively.

Based on Cs\'oka's \cite{Csoka} and P\'alvölgyi's \cite{Domotor} previous results, we already know that $v_M(n,d,k)\le v_A(n,d,k)\le\frac{k^d}{\binom{n+d-1}{d}}$. (For completeness, we include these results with their proofs, see Theorems \ref{altfelsobecs}.\ and \ref{multiplealpern}.\ ) P\'alvölgyi also showed that if $n$ is large enough compared to $k$ and $d$, and $k$ is an integer, then the values of both games are $\frac{k^d}{\binom{n+d-1}{d}}$.

\begin{pl}
    In the case $k=1$ of the Mulitple Caching game, the searcher cannot change his strategy based on the revealer's  responses, hence he has to guess the exact distribution of the treasures among the $n$ boxes. Therefore, for $k=1$ the value of the Multiple Caching Game is $\frac{1}{\binom{n+d-1}{d}}$. This argument is generalized in the proof of Theorem \ref{altfelsobecs}. 
\end{pl}

Most of our paper is concerned with the Multiple Caching Game. 
Csóka \cite{Csoka} showed that in Alpern's Caching Game, and optimal hiding strategy doesn't always place the treasures at maximal depths, this motivated the investigation of monotonicity statements in the discrete game. 
We show several such statements in Section $2$, in particular we prove the following lemma:

\begin{restatable}{lemma}{kevesebbrenemerimeg}  \label{kevesebbrenemerimeg}
In the Multiple Caching Game, for any $n,d,k$ integers such that $k\leq n$, there is an optimal searcher strategy that queries exactly $k$ boxes in each step.    
\end{restatable}

In the beginning of Section $3$, we are mainly concerned with characterizing accurate triplets: 

\begin{mydef}
We call a triplet $(n,d,k)$ \emph{accurate}, if $v_M(n,d,k)=v_A(n,d,k)=\frac{k^d}{\binom{n+d-1}{d}}$.
\end{mydef}
Csóka conjectured the following:

\begin{sejt}[Cs\'oka \cite{Csoka}]\label{n=dk}
If $n\ge dk$ and $k$ is a positive real number, then $v_A(n,d,k)=\frac{k^d}{\binom{n+d-1}{d}}$.
\end{sejt}

Pálvölgyi conjectured a stronger claim: For $n\geq dk-1$, $(n,d,k)$ is accurate.

Based on our observations and examples, we claim that even the following should hold:

\begin{sejt}\label{ujsejtes}
If $n, d, k$ are integers, and $n\ge d(k-1)+1$, then $(n, d, k)$ is accurate.
\end{sejt}

In Remark \ref{accurateconjsharp}, we show that this bound is the lowest possible.

The best known bound is due to Pálvölgyi:

\begin{thm}[P\'alvölgyi \cite{Domotor}, Theorem 2]\label{Domotoraccurate}
    If $n$ is large enough compared to $k$ and $d$, then $(n,d,k)$ is accurate.
\end{thm}

We conjecture the following monotonicity claims:
\begin{sejt}\label{pontosmonoton}
If the triplet $(n,d,k)$ of integers is accurate, then so is $(n',d',k')$ if $n' \geq n$, $d' \leq d$, and $k'\leq k$.  
\end{sejt}

As one of our main results, we settle a question of P\'alvölgyi \cite{Domotor}, where he asked whether the choice of revealed treasure in the discrete game impacts the value of the game. This problem is discussed in Chapter $\ref{rosszrandomjo}$.

\begin{restatable}{all}{altrevealer}\label{altrevealer}
When the searcher queries more than one box containing a treasure, restricting the choice of the revealed and removed item for the hider can change the game's value.
\end{restatable}

We finish Chapter 3 with examining the value of the game when the number of treasures tends to infinity.

\begin{restatable}{all}{vegtelend}\label{vegtelend}
For all integers $n,k \geq 2$ there exists a constant $c(n,k) > 0$ such that for all $d$,  $v_M(n,d,k) \geq c(n,k)$.

\end{restatable}

In Chapter $4$, we introduce different continuous versions of the game and point out a connection between one of those games, and the Manickam-Miklós-Singhi Conjecture. We also prove a theorem related to Ruckle's conjecture on accumulation games.

\section{The Multiple Caching Game}

\begin{obs} \label{segitolemma}
Revealing the location of a treasure at any point in the game does not decrease the value of the game. 
\end{obs}
\begin{proof}
If the hider tells the position of a treasure, the searcher can forget the bonus information and play following her original strategy. Therefore the value of the game cannot decrease.
\end{proof}

\begin{lemma}\label{ajandek}
Revealing the location of a treasure and removing it at any point in the game does not decrease the value of the game. 
\end{lemma}
\begin{proof}
Let us call such a treasure a \emph{gift}.
The searcher can simulate her strategy in the original game:
She can continue with the same queries as she would in the original game until she queries the location of the removed treasure. When this happens, she omits this query and instead imagines that the hider's response to this was revealing and removing the gifted treasure.

In the original game, the hider can always reveal the gifted treasure when it's first queried, and with this strategy every execution of a modified game and the corresponding (simulated) original game would be the same, therefore in the original game he can guarantee a game value not higher than that of the modified game.
\end{proof}

\begin{all} \label{dremonoton}
The value of the Multiple Caching Game is monotonically decreasing in parameter $d$: $v_M(n,d+1,k) \leq v_M(n,d,k)$ for any $n,d,k$ integers.
\end{all}
\begin{proof}
Suppose that the $(d+1)$-treasure game has value $v_M(n,d+1,k)$. For any starting hiding position, we assume that the hider gives a gifted treasure to the searcher. By  Lemma \ref{ajandek}, the value of this game is at least $v_M(n,d+1,k)$. Notice that this game is equivalent to the $(n,d,k)$-game, since the hider does not have to give away any information by hiding and immediately revealing and removing the $(d+1)$-st treasure.
\end{proof}

Now we can prove Lemma \ref{kevesebbrenemerimeg}:

\kevesebbrenemerimeg*

\begin{proof}
Let the searcher play one of her optimal strategies. We will slightly modify it:
If at any time she would query $l<k$ locations, we modify the strategy by adding some arbitrary $k-l$ locations into the query. After that, if the hider removes a treasure from the original $l$ locations, then the searcher will just continue with her strategy as if it was the original query.
If the question fails, the original question would have failed too.\\
If the hider removes a treasure from another box, then she keeps repeating the same question until she gets one from the original $l$ boxes or loses. If she loses, she would have lost with the original strategy too. If by the end of this repetition, she gets a treasure from the $l$ boxes, and $M$ additional treasures from the other $k-l$ boxes, then we claim that this new state is at least as good as the original result of the $l$-box query would be for the searcher. Imagine that for the original $l-$query, hider removes the same treasure as at the end of this process, and then gifts the other $M$ to the searcher. This does not decrease the value of the game if the searcher adapts her strategy as described in Lemma $\ref{ajandek}$, therefore there is an optimal strategy which queries the maximal amount of locations in the given step.

Applying this to each subsequent query yields an optimal strategy where the searcher queries $k$ locations in each step.

\end{proof}

\begin{kov}
When looking for an optimal strategy, it is enough to consider strategies with all questions asking exactly $k$ boxes.
\end{kov}

The following upper bound was shown by Pálvölgyi (\cite{Domotor}, included in proof of Theorem 2.).
\begin{all}[Pálvölgyi \cite{Domotor}]\label{altfelsobecs}
For any integers $n, d, k$,  $v_M(n,d,k)\le\frac{k^d}{\binom{n+d-1}{d}}$. 
\end{all}
\begin{proof}

The number of all the possible hiding variations of $d$ treasures in $n$ locations is $\binom{n+d-1}{d}$.
If the hider chooses a hiding strategy uniformly randomly amongst these, then the searcher can win with probability at most $\frac{k^d}{\binom{n+d-1}{d}}$:
Any deterministic search strategy has at most $k^{d-1}$ different possible executions, since the hider has $k$ possible answers in each step. Each such execution has $k$ distributions for which it wins, since the first $k-1$ treasures are fixed by the hider's answers and the last one has to be in of the $k$ queried locations.

\end{proof}

\begin{mj}
This upper bound holds in Alpern's Caching Game too, see Cs\'oka's paper for proof: \cite[Theorem 2.37]{Csoka}.
\end{mj}

\begin{all}[P\'alvölgyi \cite{Domotor}, Cs\'oka \cite{Csoka}]\label{multiplealpern}
If $k$ is an integer, $v_M(n,d,k)\le v_A(n,d,k)$.
\end{all}
\begin{proof} \textit{(P\'alvölgyi)}
For each initial treasure distribution of the hider in the Multiple Caching Game, we associate set of treasure distrubutions in Alpern's Caching Game: Assume the treasures distributed in the holes the same way, with arbitrary depths (such that their combined maximal depth is at most $1$.)
Given a strategy $S$ for the Multiple Caching Game, we define a searcher strategy $S'$ for Alpern's Caching Game which has greater or equal winning probability for each initial distribution of treasures.
If the next query in $S$ asks $k$ boxes, in our game the searcher will start digging in these $k$ holes simultaneously, with equal speed in all of them. As soon as she finds a treasure, she can imagine that the hider chose to reveal and remove it as a response to her query. 
If there are more treasures at the same depth, she just takes one of them, and only digs the other out later if a query asks the given hole. The searcher continues this strategy throughout the game.
If she ever starts digging in a set of holes without a treasure, then she loses the game as in the Multiple Caching Game. If she could find all treasures with this method in the discrete game for a given initial distribution of treasures, she finds them here too:
Before finding a treasure, she digs a combined depth of $k$ times the depth of the treasure before the step, therefore assuming she always picked $k$ holes containing at least one treasure, she finds all treasures after digging $k$ units (since the treasures are hidden at a combined depth of $1$).

\end{proof}

The following statement was proven in several other papers.

\begin{all} \label{k/n}
For any integers $n, d, k$,  $v_M(n,d,k)\le\frac kn$.
\end{all}
\begin{proof}
If the hider hides all treasures in the same location, then the searcher's first query finds a treasure with a probability of at most $\frac{k}{n}$, therefore the value of the game cannot be higher than this.
\end{proof}

\section{New bounds and strategies for few treasures} 
\subsection{The $(4,3,2)$ game} \label{432}

In this subsection, we introduce our notation for search strategies, while presenting an optimal search strategy for the $(n,d,k)=(4,3,2)$ case. This is the smallest non-trivial case (other than the $(3,3,2)$ game, which was already considered by Pálvölgyi \cite{Domotor}), where Pálvölgyi's conjectured bound $(n\geq dk-1)$ does not hold.

We will show that the value of the Multiple Caching Game is still equal to the upper bound here, so the bounds in Cs\'oka's and in Pálvölgyi's conjectures are not sharp.

Pálvölgyi \cite{Domotor} previously claimed that the upper bound from Theorem \ref{altfelsobecs} cannot be attained in that case by the search strategy. 
By using a new strategy, we managed to reach the theoretical maximum possible value of the game: $\frac{k^d}{\binom{n+d-1}{d}}=\frac{2^3}{\binom{4+3-1}{3}}=\frac25$.

Our strategy is as follows (see Figure \ref{(4,3,2)} for a visual representation):

\begin{itemize}
    \item In the first step, the searcher chooses between her $6$ possible queries with equal probability. We can call the two locations she picked box $1$ and box $2$. We may also assume that she was given a treasure from box $1$ (if she was given one).
    \vspace{3 mm}
    \\  For her second question, she chooses one of two options.
    \vspace{-2 mm}
    \item With $\frac45$ probability, she queries box $1$ and one of the two unqueried boxes with equal probability. We call the chosen one box $3$, the other one will be box $4$. If she receives a treasure from box $1$, then her last question is box $1$ and box $4$. If she receives treasure from box $3$, then the last question is box $3$ and box $4$.
    \item With $\frac15$ probability, she asks boxes $3$ and $4$. If she gets a treasure from box $3$, her last question is box $1$ and box $3$. If she gets a treasure from box $4$, her last question is box $1$ and box $4$.
\end{itemize}
This strategy has $\frac25$ winning probability against any hiding strategy. This strategy is a special case of the $d=3$ general case, for which the proof can be found at Theorem \ref{d=3}. 

In every strategy described in this paper, we will w.l.o.g. assume that the boxes are uniformly randomly permuted before the searcher's first query, and if the hider reveals a treasure in a previously unqueried box, then its index is minimal among such boxes.

\begin{figure}[H]
\begin{center}
\begin{tikzpicture}[line cap=round,line join=round,>=triangle 45,x=0.6cm,y=0.6cm]
\clip(-14.077056627706742,-5.148554262269194) rectangle (18.501305063321407,2.203948110660301);
\draw (-12.009185040850435,-0.7597778487318287) node[font=\small,anchor=north west] {$(1,1,0,0)$};
\draw (-5.004887851246821,1.2537458946719848) node[font=\small,anchor=north west] {$(1,0,1,0)$};
\draw (-5.004887851246821,-2.750677729850206) node[font=\small,anchor=north west] {$(0,0,1,1)$};
\draw (2.799409338356793,1.2537458946719848) node[font=\small,anchor=north west] {$(1,0,0,1)$};
\draw (2.799409338356793,-0.7597778487318287) node[font=\small,anchor=north west] {$(0,0,1,1)$};
\draw (2.7999409338356793,-2.750677729850206) node[font=\small,anchor=north west] {$(1,0,1,0)$};
\draw [->,line width=0.8pt,postaction={decorate,decoration={raise=1ex,text along path,text align=center,text={$1.$}}}] (-2.3,0.8) -- (2.8,0.8);
\draw [->,line width=0.8pt,postaction={decorate,decoration={raise=1ex,text along path,text align=center,text={$3.$}}}] (-2.3,0.8) -- (2.8,-1.15);
\draw [->,line width=0.8pt,postaction={decorate,decoration={raise=1ex,text along path,text align=center,text={$3.$}}}] (-2.3,-3.2) -- (2.8,-3.2);
\draw [line width=0.8pt] (-7,0.8) -- (-5,0.8);
\draw [line width=0.8pt] (-7,-3.2) -- (-5,-3.2);
\draw [->,line width=0.8pt,postaction={decorate,decoration={raise=1ex,text along path,text align=center,text={$1.$}}}] (-9.5,-1.2) -- (-7,-1.2);
\begin{scriptsize}
\draw[color=black] (-5.814392402105076,1.1431038289535264) node {$\frac45$};
\draw[color=black] (-5.874566818387592,-2.8839436578541007) node {$\frac15$};
\end{scriptsize}
\end{tikzpicture}
\caption{On the diagram, the positions of the $1$-s show which boxes were part of the given query. The numbers on the arrows tell us the position of the box from which the searcher received the last treasure. Where the diagram splits into several rows, we write the probability of each query occurring on a segment before the query. }
\label{(4,3,2)}
\end{center}
\end{figure}

Hence we have shown that $(4,3,2)$ is accurate.
\begin{obs}\label{Noaskback}
In this strategy, the searcher never queries a box if she has not received a treasure from it during a previous question (i.e. the box was part of the query, but the hider revealed a treasure from a different box). Therefore, if the hider ever had to choose between the boxes to reveal a treasure from, then she automatically loses, since she doesn't query the unchosen box in any later step again. Thus, the specific treasure given by the hider is irrelevant, making his choice redundant.
\end{obs}

\subsection{New Conjecture on triplet accuracy}

Based on the $(4,3,2)$-game, we stated Conjecture \ref{ujsejtes} on when should a triplet be accurate. 
Eliminating the importance of the choice of the revealed treasure limits the options of the adversary, therefore in this case the searcher could have a strategy exhibiting the property described in Observation \ref{Noaskback}.

If we consider such a strategy, each question takes $k-1$ boxes out of the game, except the last one, therefore the searcher eliminates at least $(d-1)(k-1)$ boxes in the first $(k-1)$ queries. She also needs $k$ more boxes for the last question, which gives the requirement $n\ge(d-1)(k-1)+k=d(k-1)+1$ in the conjecture.

\begin{mj}
If $n\geq d(k-1)+1$, then $\frac{k^d}{\binom{n+d-1}{d}}\le \frac kn$, hence Theorem \ref{altfelsobecs} is the stronger upper bound in these cases.
\end{mj}
\begin{proof}
    The inequality $\frac{k^d}{\binom{n+d-1}{d}}\le \frac kn$ can be simplified to $d k^{d-1} \leq \binom{n+d-1}{d-1}$. Using the lower bound on $n$, it is sufficient to show that $d k^{d-1} \leq \binom{dk}{d-1}$.
    This is clear since the left side of the inequality represents choosing $d-1$ elements from a $dk$-element set by dividing the elements into groups of size $k$ and selecting one element from $d-1$ of these groups.
\end{proof}

\begin{mj} 
Assuming Conjecture \ref{pontosmonoton}, the accuracy of the triplets $(d(k-1)+1,d,k)$ is sufficient to verify conjecture \ref{ujsejtes}.
\end{mj}

\begin{mj}\label{accurateconjsharp}
Using Theorem \ref{dremonoton} we can show that the bound on $n$ in Conjecture \ref{ujsejtes} is the best possible: By monotonicity in $d$, if $(n,d,k)$ and $(n,d-1,k)$ are accurate, then $\frac{k^d}{\binom{n+d-1}{d}}\le \frac{k^{d-1}}{\binom{n+(d-1)-1}{d-1}}$, from which  $ d(k-1)+1\le n$ follows.
\end{mj}

Now we show that Conjecture \ref{pontosmonoton} holds for $d=2,3$ by presenting optimal search strategies:

\begin{all} \label{d=2}
If $n \geq d(k-1)+1$ and $d=2$, then the triplet $(n, d, k)$ is accurate.
\end{all}
\begin{proof}
Pálvölgyi already solved the case $n\geq 2k$ in \cite{Domotor}, below we give a search strategy for $n=(2k-1)$. See Figure \ref{fig:d=2}. 

\begin{figure}[H]
\begin{center}

\begin{tikzpicture}[line cap=round,line join=round,>=triangle 45,x=0.6cm,y=0.6cm]
\clip(-14.077056627706742,-3.148554262269194) rectangle (0.501305063321407,0.203948110660301);
\draw (-14.009185040850435,-0.7597778487318287) node[font=\small,anchor=north west] {$(\underbrace{1,\ldots,1}_{k \textup{ \;}},\underbrace{0,\ldots, 0}_{k-1 \textup{ \;}})$};
\draw (-6.999409338356793,-0.7597778487318287) node[font=\small,anchor=north west] {$(1,\underbrace{0,\ldots,0}_{k-1 \textup{ \;}},\underbrace{1,\ldots,1}_{k-1 \textup{ \;}})$};
\draw [->,line width=0.8pt,postaction={decorate,decoration={raise=1ex,text along path,text align=center,text={$1.$}}}] (-9.5,-1.2) -- (-7,-1.2);
\end{tikzpicture}
\caption{We put $1$-s on the box positions that the searcher asks in her query and zeros on the other positions. The number on the arrow tells us the box from which she received a treasure.}
\label{fig:d=2}
\end{center}
\end{figure}

There are $2$ cases here, either the treasures are in the same box, or they are in $2$ different boxes.
\begin{itemize}
    \item If the treasures are in the same box, the searcher finds them with probability $\frac kn$, since she only finds them if the box is one of the first $k$ boxes.
    \item If the treasures are in different boxes, the searcher finds them with probability $\frac{k(k-1)}{\binom {n}{2}}$. This is because one treasure needs to be in the first $k$ boxes, the other one needs to be in the last $k-1$ boxes, and all the possible treasure positions are $\binom {n}{2}$.
    \item Since $n=2k-1$, the winning probability in both cases is $\frac{k(k-1)}{\binom{n}{2}}=\frac kn$.
\end{itemize}

\end{proof}

\begin{all}\label{d=3}
If $n \geq d(k-1)+1$ and $d=3$, then the triplet $(n, d, k)$ is accurate.
\end{all}
\begin{proof}
P\'alvölgyi already proved accuracy for $n\geq(3k-1)$ in \cite{Domotor}, below we give a strategy with winning probability $\frac{k^3}{\binom{n+2}{3}}$ for $n=(3k-2)$. See Figure \ref{fig:d=3}.

\begin{figure}[H]
\begin{center}
    
\scalebox{0.75}{
\begin{tikzpicture}[line cap=round,line join=round,>=triangle 45,x=0.6cm,y=0.6cm]
\clip(-14.077056627706742,-5.148554262269194) rectangle (18.501305063321407,2.203948110660301);
\draw (-14.009185040850435,-0.7597778487318287) node[font=\small,anchor=north west] {$(\underbrace{1,\ldots,1}_{k \textup{ \;}},\underbrace{0,\ldots, 0}_{2k-2 \textup{ \;}})$};
\draw (-5.004887851246821,1.2537458946719848) node[font=\small,anchor=north west] {$(1,\underbrace{0,\ldots,0}_{k-1 \textup{ \;}},\underbrace{1,\ldots,1}_{k-1 \textup{ \;}},\underbrace{0,\ldots,0}_{k-1 \textup{ \;}})$};
\draw (-5.004887851246821,-2.750677729850206) node[font=\small,anchor=north west] {$(\underbrace{0,\ldots,0}_{k \textup{ \;}},\underbrace{1,\ldots,1}_{k \textup{ \;}},\underbrace{0,\ldots,0}_{k-2 \textup{ \;}})$};
\draw (4.799409338356793,1.2537458946719848) node[font=\small,anchor=north west] {$(1,\underbrace{0,\ldots,0}_{2k-2 \textup{ \;}},\underbrace{1,\ldots,1}_{k-1 \textup{ \;}})$};
\draw (4.799409338356793,-0.7597778487318287) node[font=\small,anchor=north west] {$(\underbrace{0,\ldots,0}_{k \textup{ \;}},1,\underbrace{0,\ldots,0}_{k-2 \textup{ \;}},\underbrace{1,\ldots,1}_{k-1 \textup{ \;}})$};
\draw (4.7999409338356793,-2.750677729850206) node[font=\small,anchor=north west] {$(1,\underbrace{0,\ldots,0}_{k-1 \textup{ \;}},1,\underbrace{0,\ldots,0}_{k-1 \textup{ \;}},\underbrace{1,\ldots,1}_{k-2 \textup{ \;}})$};
\draw [->,line width=0.8pt,postaction={decorate,decoration={raise=1ex,text along path,text align=center,text={$1.$}}}] (2.3,0.8) -- (4.8,0.8);
\draw [->,line width=0.8pt,postaction={decorate,decoration={raise=1ex,text along path,text align=center,text={$(k+1).$}}}] (2.3,0.8) -- (4.8,-1.15);
\draw [->,line width=0.8pt,postaction={decorate,decoration={raise=1ex,text along path,text align=center,text={$(k+1).$}}}] (2.3,-3.2) -- (4.8,-3.2);
\draw [line width=0.8pt] (-7,0.8)-- (-5,0.8);
\draw [line width=0.8pt] (-5,-3.2)-- (-7,-3.2);
\draw [->,line width=0.8pt,postaction={decorate,decoration={raise=1ex,text along path,text align=center,text={$1.$}}}] (-9.5,-1.2) -- (-7,-1.2);
\begin{scriptsize}
\draw[color=black] (-5.814392402105076,1.5431038289535264) node {$\frac{nk^2}{\binom{n+2}{3}}$};
\draw[color=black] (-5.874566818387592,-2.4839436578541007) node {$1-\frac{nk^2}{\binom{n+2}{3}}$};
\end{scriptsize}
\end{tikzpicture}
}
\caption{On the diagram, the positions of the $1$-s show which boxes were part of the given query. The numbers on the arrows tell us the position of the box from which the searcher received the last treasure. Where the diagram splits into several rows, we write the probability of each query occurring on a segment before the query.}
\label{fig:d=3}
\end{center}
\end{figure}

Here the treasures can be positioned in the boxes in three different ways. They can all be in the same box, they can all be in separate boxes, or two can be together and the third one in a different box.
\begin{itemize}
    \item If all three treasures are in the same box $B$, the searcher finds them with probability $\frac kn\cdot\frac{nk^2}{\binom{n+2}{3}}=\frac{k^3}{\binom{n+2}{3}}$: Box $B$ has to be among the first $k$ boxes ($\frac kn$), and in the second move she must query $B$ again, which has probability $\left(\frac{nk^2}{\binom{n+2}{3}}\right)$. The final step is deterministic, so she has no more error possibility there.
    \item If $2$ treasures are in the same box, and the last one is in a different one, then there are $(3k-2)(3k-3)$ ways the treasures can be in the boxes.
    Now we count the ways the searcher can find the treasures following our strategy. The probability of each case occurring with the searcher winning is shown in parentheses at the end of its description.
  \\
    On the upper branch, she can find all treasures in the following ways:
    \begin{itemize}
        \item There are $2$ treasures in one of the first $k$ boxes and $1$ treasure in one of the last $k-1$ boxes. $\left(\frac{nk^2}{\binom{n+2}{3}}\cdot \frac{k(k-1)}{(3k-2)(3k-3)}\right)$
        \item  There is $1$ treasure in one of the first $k$ boxes, and the $2$ treasures in one of the boxes between the $k+1$-st and the $2k-1$-st. $\left(\frac{nk^2}{\binom{n+2}{3}}\cdot \frac{k(k-1)}{(3k-2)(3k-3)}\right)$
    \end{itemize}
    On the lower branch, the searcher can win in one of the following ways:
   \begin{itemize}
        \item There are $2$ treasures in one of the first $k$ boxes and $1$ treasure in one of the next $k$ boxes. $\left(\left(1-\frac{nk^2}{\binom{n+2}{3}}\right)\cdot \frac{k^2}{(3k-2)(3k-3)}\right)$
        \item  There is $1$ treasure in one of the first $k$ boxes, and the $2$ treasures in one of the next $k$ boxes. $\left(\left((1-\frac{nk^2}{\binom{n+2}{3}}\right)\cdot \frac{k^2}{(3k-2)(3k-3)}\right)$
    \end{itemize}
    Adding these up and using $n=3k-2$, we get that the winning probability of the searcher is:\\
    $$\frac1{(3k-2)(3k-3)}\cdot\left(\frac{nk^2}{\binom{n+2}{3}}\cdot2k(k-1)+\left(1-\frac{nk^2}{\binom{n+2}{3}}\right)\cdot2k^2\right)=\frac{k^3}{\binom{n+2}{3}}$$
    \item If all three treasures are in different boxes, then the number of ways the hider could distribute them is $\binom n3$.
    
    On the upper branch, the searcher can find the treasures if one treasure is hidden in one of the first $k$ boxes, another one is hidden in the next $k-1$ boxes, and one is hidden in the last $k-1$ boxes. $\left(\frac{nk^2}{\binom{n+2}{3}}\cdot \frac{k(k-1)^2}{\binom n3}\right)$

    On the lower branch, the searcher can find them, if there is a treasure in the first $k$ boxes, a treasure in the next $k$ boxes, and one in the last $k-2$ boxes. $\left(\left(1-\frac{nk^2}{\binom{n+2}{3}}\right)\cdot \frac{k^2(k-2)}{\binom n3}\right)$

     Adding these up and using $n=3k-2$, we get that the winning probability of the searcher is:\\
     $$\frac{1}{\binom{n}{3}}\cdot\left(\frac{nk^2}{\binom{n+2}{3}}\cdot k(k-1)^2+\left(1-\frac{nk^2}{\binom{n+2}{3}}\right)\cdot k^2(k-2)\right)=\frac{k^3}{\binom{n+2}{3}}$$
\end{itemize}

Hence we proved that the value of the game is $\frac{k^3}{\binom{n+2}{3}}$.
\end{proof}
We verified our conjecture for a few more cases:
$(5,4,2),(9,4,3),(13,4,4),(6,5,2)$. We include the strategy for the $(5,4,2)$ game, see Figure \ref{fig:(5,4,2)}. We leave it to the reader to check its correctness.

\begin{figure}[H]
\begin{center}
\begin{tikzpicture}[line cap=round,line join=round,>=triangle 45,x=0.6cm,y=0.6cm]
\clip(-0.1,-8) rectangle (30,6);
\draw (0,0) node[font=\small,anchor=north west] {$(1,1,0,0,0)$};
\draw (6,2) node[font=\small,anchor=north west] {$(1,0,1,0,0)$};
\draw (6,-4) node[font=\small,anchor=north west] {$(0,0,1,1,0)$};
\draw (12,4) node[font=\small,anchor=north west] {$(1,0,0,1,0)$};
\draw (12,2) node[font=\small,anchor=north west] {$(0,0,1,1,0)$};
\draw (12,-2) node[font=\small,anchor=north west] {$(0,0,1,0,1)$};
\draw (12,-4) node[font=\small,anchor=north west] {$(1,0,0,0,1)$};
\draw (12,-6) node[font=\small,anchor=north west] {$(1,0,1,0,0)$};

\draw (18,5) node[font=\small,anchor=north west] {$(1,0,0,0,1)$};
\draw (18,4) node[font=\small,anchor=north west] {$(0,0,0,1,1)$};
\draw (18,3) node[font=\small,anchor=north west] {$(0,0,1,0,1)$};
\draw (18,2) node[font=\small,anchor=north west] {$(0,0,0,1,1)$};
\draw (18,-1) node[font=\small,anchor=north west] {$(1,0,1,0,0)$};
\draw (18,-2) node[font=\small,anchor=north west] {$(1,0,0,0,1)$};
\draw (18,-3) node[font=\small,anchor=north west] {$(1,0,1,0,0)$};
\draw (18,-4) node[font=\small,anchor=north west] {$(0,0,1,0,1)$};
\draw (18,-5) node[font=\small,anchor=north west] {$(1,0,0,0,1)$};
\draw (18,-6) node[font=\small,anchor=north west] {$(0,0,1,0,1)$};

\draw [->,line width=0.8pt,postaction={decorate,decoration={raise=1ex,text along path,text align=center,text={$1.$}}}] (3,-0.5) -- (4.5,-0.5);
\draw [->,line width=0.8pt,postaction={decorate,decoration={raise=1ex,text along path,text align=center,text={$1.$}}}] (9,1.5) -- (12,3.5);
\draw [->,line width=0.8pt,postaction={decorate,decoration={raise=1ex,text along path,text align=center,text={$3.$}}}] (9,1.5) -- (12,1.5);
\draw [->,line width=0.8pt,postaction={decorate,decoration={raise=1ex,text along path,text align=center,text={$3.$}}}] (9,-4.5) -- (10.5,-4.5);

\draw [->,line width=0.8pt,postaction={decorate,decoration={raise=1ex,text along path,text align=right,text={$1.$}}}] (15,3.5) -- (18,4.5);
\draw [->,line width=0.8pt,postaction={decorate,decoration={raise=1ex,text along path,text align=right,text={$4.$}}}] (15,3.5) -- (18,3.5);
\draw [->,line width=0.8pt,postaction={decorate,decoration={raise=1ex,text along path,text align=right,text={$3.$}}}] (15,1.5) -- (18,2.5);
\draw [->,line width=0.8pt,postaction={decorate,decoration={raise=1ex,text along path,text align=right,text={$4.$}}}] (15,1.5) -- (18,1.5);
\draw [->,line width=0.8pt,postaction={decorate,decoration={raise=1ex,text along path,text align=right,text={$3.$}}}] (15,-2.5) -- (18,-1.5);
\draw [->,line width=0.8pt,postaction={decorate,decoration={raise=1ex,text along path,text align=right,text={$5.$}}}] (15,-2.5) -- (18,-2.5);
\draw [->,line width=0.8pt,postaction={decorate,decoration={raise=1ex,text along path,text align=right,text={$1.$}}}] (15,-4.5) -- (18,-3.5);
\draw [->,line width=0.8pt,postaction={decorate,decoration={raise=1ex,text along path,text align=right,text={$5.$}}}] (15,-4.5) -- (18,-4.5);
\draw [->,line width=0.8pt,postaction={decorate,decoration={raise=1ex,text along path,text align=right,text={$1.$}}}] (15,-6.5) -- (18,-5.5);
\draw [->,line width=0.8pt,postaction={decorate,decoration={raise=1ex,text along path,text align=right,text={$3.$}}}] (15,-6.5) -- (18,-6.5);

\draw [line width=0.8pt] (4.7,1.5) -- (6,1.5);
\draw [line width=0.8pt] (4.7,-4.5) -- (6,-4.5);
\draw [line width=0.8pt] (10.7,-2.5) -- (12,-2.5);
\draw [line width=0.8pt] (10.7,-4.5) -- (12,-4.5);
\draw [line width=0.8pt] (10.7,-6.5) -- (12,-6.5);

\begin{scriptsize}
\draw[color=black] (5.35,1.9) node {$\frac47$};
\draw[color=black] (5.35,-4.1) node {$\frac37$};
\draw[color=black] (11.35,-2.1) node {$\frac13$};
\draw[color=black] (11.35,-4.1) node {$\frac13$};
\draw[color=black] (11.35,-6.1) node {$\frac13$};

\end{scriptsize}
\end{tikzpicture}

\caption{On the diagram, the positions of the $1$-s show which boxes were part of the given query. The numbers on the arrows tell us the position of the box from which the searcher received the last treasure. Where the diagram splits into several rows, we write the probability of each query occurring on a segment before the query.}
\label{fig:(5,4,2)}
\end{center}
\end{figure}

\subsection{Adversary, Cooperative, and Random revealer} \label{rosszrandomjo}
After arguing strategies in the last section, one could be curious to see if it is significant, whether the hider can choose which treasure he gives reveals and removes. This motivates the definition of the following two new versions of the Multiple Caching Game:

The game itself does not change in any of the cases, the only thing that changes is how the hider chooses the treasure to give away when multiple treasures are found in a single question.

\begin{mydef}
If the searcher guesses more treasures in a single question, the \textbf{Adversary revealer} chooses one to reveal, following the hider's strategy to win the game.
This is the original version of the Discrete Caching Game.

We still denote the value of this game by $v_M(n,d,k)$.
\end{mydef}

\begin{mydef}
If the searcher guesses more treasures in a single question, the \textbf{Random revealer} reveals one of them with equal probability. 

We denote the value of this game by $v_{Mr}(n,d,k)$.
\end{mydef}

\begin{mydef}
If the searcher guesses more treasures in a single question, the \textbf{Cooperative revealer} reveals a treasure, following a strategy discussed with the searcher in advance.

We denote the value of this game by $v_{Mc}(n,d,k)$.
\end{mydef}

\begin{all}
For any integer triplet $(n,d,k)$,  $$v_M(n,d,k)\leq v_{Mr}(n,d,k)\leq v_{Mc}(n,d,k)\leq\frac{k^d}{\binom{n+d-1}{d}}.$$
\end{all}
\begin{proof} We prove the three inequalities in order from right to left.

\begin{itemize}
    \item In the proof of Theorem \ref{altfelsobecs}, we only used the hiding strategy and never mentioned how the hider chooses which treasure to reveal. Therefore, the upper bound holds for all defined versions of the game, even with the cooperative revealer.
    \item $v_{Mr}(n,d,k)\leq v_{Mc}(n,d,k)$, since the cooperative revealer can apply the random treasure-giving strategy, which shows that she can get the same winning probability as in the random revealer case. 
    \item $v_M(n,d,k)\leq v_{Mr}(n,d,k)$, since the adversary revealer can choose to follow the random revealer strategy, which gives him a maximal winning chance at least as large as the random revealer has. Therefore, the value of the random revealer game is at least as large as the adversary revealer game. 

\end{itemize}
\end{proof}

Note that this theorem implies that for accurate $(n,d,k)$ triplets, $v_M(n,d,k) = v_{Mr}(n,d,k) = v_{Mc}(n,d,k)\leq\frac{k^d}{\binom{n+d-1}{d}}$, so for a triplet exhibiting a gap between the values of the three variants, we need to find a non-accurate triplet.

Such a case is $(n,d,k)=(3,3,2)$, which we discuss below, proving Theorem \ref{altrevealer}.

The optimal hiding (and revealing) strategies for the adversary and random revealer games were found using a computer program.  If we make choice variables for every single choice the hider, the searcher or the revealer has during the game (representing the probabilities of their answers in the current state of the game), then we can express the game's value using these variables; an optimal hiding strategy is for which the maximum of this expression among the possible search strategies is as low as possible. For the random and adversary revealer cases, we give a full proof of this upper bound in Appendix \ref{Appendix}.  The upper bound of the cooperative revealer case follows directly from Theorem \ref{k/n}.

We got the following results:
\begin{equation*}
    v_{Mc}(3,3,2)= \frac{2}{3}, \; \;
 v_{Mr}(n,d,k)= \frac{12}{19}, \; \;
v_{M}(n,d,k)= \frac{3}{5}.
\end{equation*}

Below we present searching strategies for all three cases:

Figure \ref{fig:kind} shows the optimal strategy of the searcher in the cooperative hider game. Here we have shown that $v_{Mc}(3,3,2)=\frac23$.

\begin{figure}[H]
\begin{center}
\begin{tikzpicture}[line cap=round,line join=round,>=triangle 45,x=0.6cm,y=0.6cm]
\clip(-0.1,-2) rectangle (15,2);
\draw (0,0) node[font=\small,anchor=north west] {$(1,1,0)$};
\draw (5,0) node[font=\small,anchor=north west] {$(1,1,0)$};
\draw (10,1) node[font=\small,anchor=north west] {$(1,0,1)$};
\draw (10,-1) node[font=\small,anchor=north west] {$(0,1,1)$};

\draw [->,line width=0.8pt,postaction={decorate,decoration={raise=1ex,text along path,text align=center,text={$1.$}}}] (2,-0.5) -- (5,-0.5);
\draw [->,line width=0.8pt,postaction={decorate,decoration={raise=1ex,text along path,text align=center,text={$1.$}}}] (7,-0.5) -- (10,0.5);
\draw [->,line width=0.8pt,postaction={decorate,decoration={raise=1ex,text along path,text align=center,text={$2.$}}}] (7,-0.5) -- (10,-1.5);

\end{tikzpicture}

\caption{On the diagram, the positions of the $1$-s show which boxes were part of the given query. The numbers on the arrows tell us the position of the box from which the searcher received the last treasure.}
\label{fig:kind}
\end{center}
\end{figure}

The searcher has $\frac 23$ winning probability if she applies that strategy. Since the revealer is cooperative, the searcher tells him to give a treasure from the box which has the least amount of treasures, if both queried boxes have treasures in them.

Keeping this in mind, we show that this is indeed an optimal strategy for the searcher:
\begin{itemize}
    \item If the hider places all the treasures in the same box, then the searcher has $\frac23$ probability of asking it in the first question, and in this case she queries this box two more times, finding all treasures. This gives $\frac23$ winning chance.
    \item If the hider puts two treasures in the same box, and the last one in another, there are three cases:
    \begin{itemize}
    \item With $\frac 13$ chance, one of the first two boxes contains $1$ treasure, and the other one contains $2$ treasures. By the cooperative revealer's strategy, the searcher gets the sole treasure from the first box. Then it is easy to see that she will ask the second box $2$ more times.
    \item With $\frac 13$ chance, box one has $2$ treasures and box three has the last treasure. In this case, the searcher surely wins again.
    \item With $\frac 13$ chance, box one has $1$ treasures and box three has the other $2$ treasures. In this case, the searcher never wins.
    \end{itemize}
    Altogether, we can see that the searcher wins with $\frac 23$ probability in this case too.
    \item If the hider hides all the treasures in different boxes, then following the diagram one can see that the searcher will always find all $3$ treasures. So she surely wins in that case.
\end{itemize}
Therefore we can see that this strategy has $\frac 23$ winning probability even in the worst case. It is an optimal strategy since the upper bound of Theorem \ref{k/n} holds.

In the random revealer case we found the following strategy which exhibits a winning probability of $\frac{12}{19} \approx 0.631579$ for all three hiding strategies:

\begin{figure}[H]\label{Random332abra}
\begin{center}
\begin{tikzpicture}[line cap=round,line join=round,>=triangle 45,x=0.6cm,y=0.6cm]
\draw (0,1) node[font=\small,anchor=north west] {$(1,1,0)$};

\draw (8,3) node[font=\small,anchor=north west] {$(1,0,1)$};
\draw (8,-1) node[font=\small,anchor=north west] {$(0,1,1)$};

\draw (16,4) node[font=\small,anchor=north west] {$(1,1,0)$};

\draw (16,2) node[font=\small,anchor=north west] {$(1,0,1)$};
\draw (16,1) node[font=\small,anchor=north west] {$(0,1,1)$};

\draw (16,-0.5) node[font=\small,anchor=north west] {$(1,1,0)$};

\draw (16,-1.5) node[font=\small,anchor=north west] {$(1,0,1)$};

\draw [->,line width=0.8pt,postaction={decorate,decoration={raise=1ex,text along path,text align=center,text={$1.$}}}] (2,0.5) -- (5,0.5);

\draw [->,line width=0.8pt,postaction={decorate,decoration={raise=1ex,text along path,text align=center,text={$1.$}}}] (10,2.5) -- (13,3.5);
\draw [->,line width=0.8pt,postaction={decorate,decoration={raise=1ex,text along path,text align=center,text={$2.$}}}] (10,2.5) -- (13,1);

\draw [->,line width=0.8pt,postaction={decorate,decoration={raise=1ex,text along path,text align=center,text={$2.$}}}] (10,-1.5) -- (13,-1);
\draw [->,line width=0.8pt,postaction={decorate,decoration={raise=1ex,text along path,text align=right,text={$3.$}}}] (10,-1.5) -- (13,-2);

\draw [line width=0.8pt] (5.5,2.5) to[out=0, in=180] node[pos=0.5, above=0.5ex] {$\frac{18}{19}$} (7.5,2.5);
\draw [line width=0.8pt] (5.5,-1.5) to[out=0, in=180] node[pos=0.5, above=0.5ex] {$\frac{1}{19}$}  (7.5,-1.5);

\draw [line width=0.8pt] (13.5,1.5) to[out=0, in=180] node[pos=0.5, above=-0.4ex] {$\frac{1}{3}$} (15.5,1.5);
\draw [line width=0.8pt] (13.5,0.5) to[out=0, in=180] node[pos=0.5, above=-0.4ex] {$\frac{2}{3}$} (15.5,0.5);


\end{tikzpicture}

\caption{Search strategy when the revealer is randomized}
\end{center}
\end{figure}

In the following search strategy for the adversary revealer case, the choice of the revealed treasure is never meaningful, so it is sufficient to check that the probability of winning is $\frac{6}{10}$ for all $3$ possible hiding allocations.
\begin{figure}[H]
\begin{center}
\begin{tikzpicture}[line cap=round,line join=round,>=triangle 45,x=0.6cm,y=0.6cm]
\draw (0,0) node[font=\small,anchor=north west] {$(1,1,0)$};

\draw (8,4) node[font=\small,anchor=north west] {$(1,1,0)$};
\draw (8,0) node[font=\small,anchor=north west] {$(1,0,1)$};
\draw (8,-4) node[font=\small,anchor=north west] {$(0,1,1)$};

\draw (16,4.5) node[font=\small,anchor=north west] {$(1,0,1)$};

\draw (16,3.5) node[font=\small,anchor=north west] {$(0,1,1)$};

\draw (16,2) node[font=\small,anchor=north west] {$(1,1,0)$};
\draw (16,1) node[font=\small,anchor=north west] {$(1,0,1)$};

\draw (16,-1) node[font=\small,anchor=north west] {$(1,0,1)$};
\draw (16,-2) node[font=\small,anchor=north west] {$(0,1,1)$};

\draw (16,-3.5) node[font=\small,anchor=north west] {$(1,1,0)$};

\draw (16,-4.5) node[font=\small,anchor=north west] {$(1,0,1)$};

\draw [->,line width=0.8pt,postaction={decorate,decoration={raise=1ex,text along path,text align=center,text={$1.$}}}] (2,-0.5) -- (5,-0.5);

\draw [->,line width=0.8pt,postaction={decorate,decoration={raise=1ex,text along path,text align=center,text={$1.$}}}] (10,3.5) -- (13,4);
\draw [->,line width=0.8pt,postaction={decorate,decoration={raise=1ex,text along path,text align=right,text={$2.$}}}] (10,3.5) -- (13,3);

\draw [->,line width=0.8pt,postaction={decorate,decoration={raise=1ex,text along path,text align=center,text={$1.$}}}] (10,-0.5) -- (13,1);
\draw [->,line width=0.8pt,postaction={decorate,decoration={raise=1ex,text along path,text align=center,text={$3.$}}}] (10,-0.5) -- (13,-2);

\draw [->,line width=0.8pt,postaction={decorate,decoration={raise=1ex,text along path,text align=center,text={$2.$}}}] (10,-4.5) -- (13,-4);
\draw [->,line width=0.8pt,postaction={decorate,decoration={raise=1ex,text along path,text align=right,text={$3.$}}}] (10,-4.5) -- (13,-5);

\draw [line width=0.8pt] (5.5,3.5) to[out=0, in=180] node[pos=0.5, above=0.5ex] {$\frac{3}{10}$} (7.5,3.5);
\draw [line width=0.8pt] (5.5,-0.5) to[out=0, in=180] node[pos=0.5, above=0.5ex] {$\frac{6}{10}$} (7.5,-0.5);
\draw [line width=0.8pt] (5.5,-4.5) to[out=0, in=180] node[pos=0.5, above=0.5ex] {$\frac{1}{10}$}  (7.5,-4.5);

\draw [line width=0.8pt] (13.5,1.5) to[out=0, in=180] node[pos=0.5, above=-0.4ex] {$\frac{1}{2}$} (15.5,1.5);
\draw [line width=0.8pt] (13.5,0.5) to[out=0, in=180] node[pos=0.5, above=-0.4ex] {$\frac{1}{2}$} (15.5,0.5);

\draw [line width=0.8pt] (13.5,-1.5) to[out=0, in=180] node[pos=0.5, above=-0.4ex] {$\frac{1}{2}$}(15.5,-1.5);
\draw [line width=0.8pt] (13.5,-2.5) to[out=0, in=180] node[pos=0.5, above=-0.4ex] {$\frac{1}{2}$} (15.5,-2.5);

\end{tikzpicture}

\caption{Search strategy when the revealer is an adversary}
\end{center}
\end{figure}

\subsection{Lower bound for $v_M(n,d,k)$ as $d \to \infty$}
We proved that $v_M(n,d,k)$ is monotonically decreasing with respect to $d$. This raises the question whether the value of the game tends to $0$ as $d\to\infty$. We show that the value does not converge to $0$ if $k \geq 2$, so the searcher can find arbitrarily many treasures with a fixed positive probability. 

\vegtelend*
\begin{proof}
Fix $n$ and $k \geq 2$. We define a strategy that finds all the treasures with a probability independent of $d$: In the first step, the searcher queries $k$ arbitrarily chosen boxes. In all subsequent steps, she queries the box she received the treasure from in the previous step and asks $k-1$ more boxes chosen uniformly randomly out of all the other boxes. In the first step, she loses with probability at most $\frac{n-k}n$ since there must be at least one box containing a treasure. In a general step, the searcher can only lose if the last treasure she received was the last in its box, therefore during the game, there are at most $n$ steps in which she loses with positive probability. If before such a step there were $i-1$ empty boxes, and afterwards there are $i$ empty boxes, then in the next step the searcher loses with probability $\frac{\binom{i-1}{k-1}}{\binom{n-1}{k-1}}$, since she guesses $k-1$ random boxes out of all the remaining ones. Therefore the winning probability of the searcher with this strategy is at least \\ $c(n,k)=\left(1-\frac{n-k}n \right)\cdot\prod\limits_{i=k}^{n-1}\left( 1-\frac{\binom{i-1}{k-1}}{\binom{n-1}{k-1}} \right)>0$.
\end{proof}
\begin{kov}
For any $k \geq 2$, $v_A(n,d,k) \geq c(n,\lfloor k \rfloor)>0$, where $c(n,\lfloor k \rfloor)$ is a positive constant only depending on $n$ and $\lfloor k \rfloor$.
\end{kov}
\begin{proof}
Using Theorem \ref{multiplealpern}.\, for any integer $k$, $v_A(n,d,k) \ge c(n,k)$. For non-integer $k$ values, $v_A(n,d,k) \ge v_A(n,d,\lfloor k \rfloor)\ge c(n,\lfloor k \rfloor)$ using the game's definition.
\end{proof}

\section{Continuous versions}

Alpern's Caching Game can be regarded as a continuous version of the Multiple Caching Game, this motivates the definition of other such games, replacing a given parameter of the game with a continuous variable. 
\subsection{Discrete game for real $k$ values} \label{nemegeszk}
We can define the multiple caching game for non-integer values of $k$. The searcher is allowed to use a probabilistic strategy, giving a probability for each possible query, such that the expected number of chosen boxes in each step is $k$. 

In this subsection, we assume $k$ to be any positive real number.

\begin{all}
    The upper bound $\frac{k^d}{\binom{n+d-1}{d}}$ for the value of the Discrete Caching Game remains true for arbitrary positive real $k$ values.
\end{all}
\begin{proof} 
    We claim that the hider can achieve this game value using the same hiding strategy as before, choosing an allocation uniformly randomly.
    
    For an arbitrary search strategy, let $\mathcal{Q}_1$ denote the set of first queries she makes with positive probability, and for query $Q_1=\{h_1^1, h_1^2, \dots, h_1^{|Q_1|}\} \in \mathcal{Q}_1$, let $\mathcal{Q}_2^{(Q_1,{j_1})}$ denote the set of second queries she makes with positive probability if the hider revealed a treasure from box $j_1 \in Q_1$ in step 1.
    Similarly, for step $s$ of the game we define $\mathcal{Q}_s^{(Q_1,{j_1}), (Q_2,{j_2}), \dots, (Q_{s-1},{j_{s-1}})}$ to  be the set of positive probability queries of the searcher if the first $s-1$ queries of the game were $Q_1 \in \mathcal{Q}_1, Q_2\in \mathcal{Q}_2^{(Q_1,j_1) }, \dots, Q_{s-1} \in {\mathcal{Q}_{s-1}^{(Q_1,j_1), (Q_2,j_2), \dots, (Q_{s-2},j_{s-2})}}$, and the first $s-1$ reveals were from boxes $j_1, j_2, \dots j_{s-1}$.

    Using these notations, the expected number of configurations where the searcher wins is at most 
    
    \begin{multline*}
        \sum\limits_{Q_1 \in \mathcal{Q}_1} P(Q_1) \sum\limits_{j_1 \in Q_1} \dots \sum\limits_{Q_{d-1} \in \mathcal{Q}_{d-1}^{(Q_1,j_1), \dots, (Q_{d-2},j_{d-2})}} P\left({Q}_{d-1}\right) \sum\limits_{j_{d-1} \in Q_{d-1}}   \sum\limits_{Q_d \in \mathcal{Q}_d^{(Q_1,j_1), \dots, (Q_{d-1},j_{d-1})}} P(Q_d) \sum\limits_{j_d \in Q_d} 1 =  \\
        \sum\limits_{Q_1 \in \mathcal{Q}_1} P(Q_1) \sum\limits_{j_1 \in Q_1} \dots \sum\limits_{Q_{d-1} \in \mathcal{Q}_{d-1}^{(Q_1,j_1), \dots, (Q_{d-2},j_{d-2})}} P\left({Q}_{d-1}\right) \sum\limits_{j_{d-1} \in Q_{d-1}}   E(|Q_d|) =  \\
        \sum\limits_{Q_1 \in \mathcal{Q}_1} P(Q_1) \sum\limits_{j_1 \in Q_1} \dots \sum\limits_{Q_{d-1} \in \mathcal{Q}_{d-1}^{(Q_1,j_1), \dots, (Q_{d-2},j_{d-2})}} P\left({Q}_{d-1}\right) \sum\limits_{j_{d-1} \in Q_{d-1}}   k = \dots = k^d.
    \end{multline*}

\end{proof}

\begin{all}
If $n$ is large enough compared to $d$ and $k$, then the value of the game is $\frac{k^d}{\binom{n+d-1}{d}}$ for any positive real $k$.
\end{all}

\begin{proof}
We already have the upper bound from the previous theorem. 

For the lower bound, we use a search strategy similar to the one in the paper of Pálvölgyi. \cite[Proof of Theorem 2.]{Domotor}.
In each step, the searcher will ask $\lfloor k\rfloor$ boxes with probability $p$ and $\lceil k\rceil$ boxes with probability $1-p$ such that the expected value is $k$.
We define $p_\lambda (n,d,1)$ the same way P\'alvölgyi did in his paper. 

For this, we need to revisit the strategy of the searcher if she can only ask 1 box in each step. Then her strategy is the following:

She picks randomly a starting allocation $\mu$ of the treasures, uniformly distributed among all the $\binom{n+d-1}{d}$ possibilities. Then she will ask the box with the most treasures in $\mu$. She asks it as many times, as many treasures there are in that box in $\mu$.
Then she asks the box with the second most treasures and asks just enough times to make it empty assuming $\mu$. She repeats that until she asked all her questions. If several boxes have the same amount of treasures, she picks one of them with equal probability. This way, if we count the discovered treasures in the boxes she picked in order, we get a Young diagram $\lambda$ at any point of the game.
For the outside observer, who does not know $\mu$, the Young diagram $\lambda$ tells a probability of the searcher asking the same box again in the next question (provided that the observer knows the strategy, but does not know the $\mu$, he can calculate conditional probabilities of all possible Young diagram shapes). This probability with a given Young table $\lambda$ is called $p_\lambda (n,d,1)$.
For any positive integer $s$, we define $p_\lambda (n,d,s)=s\cdot p_\lambda (n,d,1)$.

We have to assume that $n$ is large enough, such that $p_\lambda (n,d,\lceil k\rceil)\leq 1$, and also such that $n\geq d\cdot \lceil k\rceil$.

Now we can define the strategy of the searcher. The queries are asked such that the discovered treasures form a Young diagram at any point throughout the game. The searcher will ask:
\begin{itemize}
\item The box she received the last treasure from, and $\lfloor k\rfloor-1$ brand new, previously untouched boxes with $p\cdot p_\lambda (n,d,\lfloor k\rfloor)$ probability.
\item $\lfloor k\rfloor$ brand new, previously untouched boxes with $p\cdot (1-p_\lambda (n,d,\lfloor k\rfloor))$ probability.
\item The box she received the last treasure from, and $\lceil k\rceil-1$ brand new, previously untouched boxes with $(1-p)\cdot p_\lambda (n,d,\lceil k\rceil)$ probability.
\item $\lceil k\rceil$ brand new, previously untouched boxes with $(1-p)\cdot (1-p_\lambda (n,d,\lceil k\rceil))$ probability.
\end{itemize}

We remark that if $k$ is an integer, then this is the same strategy as the one in Pálvölgyi's paper.

Now we need to show that this strategy wins with $\frac{k^d}{\binom{n+d-1}{d}}$ probability. If any question asks more than one box that contains treasures, the searcher has already lost. Therefore, it is enough to show that with this strategy, each step has $k$ times the probability of finding the box $i$ with a treasure than the corresponding step in the $k=1$ case.

Here we have 2 cases to consider: If box $i$ is the box she last received a treasure from, then if $k=1$, she had a probability $p_\lambda (n,d,1)$ to find it. In the current strategy, her chance to find the box is: \\
$p\cdot p_\lambda (n,d,\lfloor k\rfloor)+(1-p)\cdot p_\lambda (n,d,\lceil k\rceil)=p_\lambda (n,d,1)\cdot (p\lfloor k\rfloor +(1-p)\lceil k\rceil)=k\cdot p_\lambda (n,d,1)$ \\
which we wanted to show.

If box $i$ is a new box and $k=1$ she had probability $1-p_\lambda (n,d,1)$ to find it.
In the current strategy, her chance to find the box is: \\
$(\lfloor k\rfloor-1)\cdot p\cdot p_\lambda (n,d,\lfloor k\rfloor)+\lfloor k\rfloor \cdot p\cdot (1-p_\lambda (n,d,\lfloor k\rfloor))+(\lceil k\rceil-1)\cdot (1-p)\cdot p_\lambda (n,d,\lceil k\rceil)+\lceil k\rceil \cdot (1-p)\cdot (1-p_\lambda (n,d,\lceil k\rceil))=\\(\lfloor k\rfloor p+\lceil k\rceil (1-p))-(p\cdot p_\lambda (n,d,\lfloor k\rfloor)+(1-p)\cdot p_\lambda (n,d,\lceil k\rceil))=k\cdot (1-p_\lambda (n,d,1))$ \\
which we wanted to show. 

If $k$ was $1$, the searcher could win the game with probability $\frac{1}{\binom{n+d-1}{d}}$. Therefore since the game consists of $d$ steps, and the chance of discovering a new treasure in each step is $k$ times the chance of discovering it with $1$-sized queries, we get that the winning probability of the searcher is $\frac{k^d}{\binom{n+d-1}{d}}$.
\end{proof}

\subsection{Accumulation game} \label{mogyorokrem}

In this section, we demonstrate a version of the Multiple Caching Game where the amount of treasure can be an arbitrary positive real number. In this game, the hider has positive $d$ weight units of gold which he can distribute among $n$ locations. The searcher can query a set of $k$ locations and he loses if there is less than $1$ unit of gold in these locations, otherwise he receives one unit of gold from these.
Even for integer $d$, the hider may distribute the gold in a way that makes it impossible for the searcher to find all $d$ units, therefore we consider the one-turn version of the game, i.e. the searcher wins if the $k$ boxes of his first query contain at least $1$ unit of gold altogether.
\\
The goal of hider is to minimize the number of $k$-sets of locations with combined weight at least $1$.
\begin{obs}\label{sokgold}
    If $d \geq n$, then there will surely be a location with more than $1$ unit of gold, therefore it is optimal for the hider to place all gold in one location, so in this case the value of the game is $\frac{k}{n}$.
\end{obs}
\begin{obs}\label{kevesgold}
    If $d < \frac{n}{k}$, then the hider can distribute the gold evenly among the locations, giving the searcher $0$ probability of success.
\end{obs}

An equivalent version of this game was previously considered by Ruckle \cite{Ruckle1, Ruckle2}, by Ruckle and Kikuta \cite{Ruckle3}, and by Alpern, Fokkink and Kikuta \cite{ujalpern} \footnote{In their setting, the hider wants to maximize the number of $k$-sets of locations which contain at least $1$ unit altogether. However, easy consideration shows that the validity of Ruckle's conjecture is equivalent to that in our case.}.
Ruckle stated the following conjecture regarding the optimal hiding placement of the hider:
\begin{sejt}[Ruckle's Conjecture] \cite{Ruckle2}
For all parameters, there exists an optimal hiding strategy of the following form: The hider chooses $r$ boxes and hides exactly $\frac dr$ units of gold in each of them. The rest of the boxes remain empty.
\end{sejt}

They verified the conjecture in several special cases, for example when $k \leq 2$ or $n-k \leq 2$, and when $k\mid n$. We include a short alternate proof for the last special case.

\begin{all}\label{bar}
If $k\mid n$, the hider wins with probability at most $\frac{k}{n}$ for any  $d\geq \frac{n}{k}$. 
\end{all}
\begin{proof}[Proof due to Andr\'as Imolay, personal communication.]
Using the Baranyai Theorem \cite{sajt}, we can partition the set of all $k$-element subsets of $[n]$ in a way such that each element of the partition contains $\frac{n}{k}$ subsets, whose disjoint union is precisely $[n]$. Since altogether there are at least $\frac{n}{k}$ units of gold in the boxes, in each element of the partition at least one $k$-element set must contain at least $1$ unit of gold by the pigeonhole principle. Therefore, at least  $\frac{k}{n}$ fraction of all $k$-element sets have at least $1$ unit of gold, so the hider wins with a probability at most $1-\frac{k}{n}$.
The lower bound comes from simply hiding all of the gold in the same box.
\end{proof}

Theorem \ref{bar} implies that in the case $k \mid n$, it is optimal for the hider to either hide all the gold in one location or distribute it evenly among them, verifying Ruckle's conjecture for this case. 
For other values of  $k$ and $n$, there may exist values for $d$ for which the optimal hiding strategy is non-trivial:

\begin{pl}
If $n=5$, $k=3$ and $\frac53 \leq d < 3$, the best hiding strategy gives a chance of $\frac{7}{10}$ for the hider to win.
\end{pl}
\begin{proof}
If the hider places $\frac{d}{3}$ units in $3$ boxes, and leaves the others empty, the searcher can only win if she guesses $2$ of these $3$ boxes. She has $\frac{3}{10}$ chance to do so. So the hider wins with probability $\frac{7}{10}$ in this case. Note that this distribution still satisfies the property described in Ruckle's conjecture.

Now we show that there isn't a distribution of gold for which at least 8 box triples contain less than 1 unit combined.
Suppose there is such a hiding allocation. Let's say that the boxes contain $a_1 \geq a_2 \geq a_3 \geq a_4 \geq a_5$ weight units of gold respectively. Since at most $2$ of the triple-wise sums add up to at least $1$ unit, these sums must be among  $a_1+a_2+a_3$ and $a_1+a_2+a_4$.
Hence, the other sums must be less than $1$, which means $a_1+a_3+a_4, a_2+a_3+a_4, a_3+a_4+a_5, a_1+a_2+a_5$ are all less than $1$. We can create a linear combination of them to get: $3(a_1+a_2+a_3+a_4+a_5) < 5$. Consequently, $d=a_1+a_2+a_3+a_4+a_5 < \frac{5}{3}$, which is a contradiction.
\end{proof}

Ruckle's conjecture is quite similar to the Manickam-Mikl\'os-Singhi conjecture \cite{Manickam}. This similarity was previously discovered by Csóka \cite{Csoka}, we point out an explicit connection between them.

\begin{kerd}
Let $n,k$ be fixed integers. Let $\{i_1, i_2, \dots, i_k \} \subset \{ 1,2, \dots, n\}$ be a $k$-element subset, chosen uniformly. Try to find $a_1,a_2, \dots, a_n$ real numbers, such that \begin{equation}\label{MMSeq}
    P\left( a_{i_1} + a_{i_2} + \dots + a_{i_k} < \frac{k}{n} \sum\limits_{i=1}^n a_i \right) 
\end{equation} is maximized. 
\end{kerd}
\begin{sejt}[Manickam–Mikl\'os–Singhi]
If $4k \leq n$, then the maximum probability is $1 - \frac{k}{n}$.
\end{sejt}

Let us denote the weight of the gold in each corresponding box as $a_1,a_2,\dots ,a_n$, and their sum as $d$. If $\frac{k}{n} d =1$, then the MMS conjecture implies that the hider cannot win with probability greater than $1-\frac{k}{n}$. Combining this with Observation \ref{kevesgold} shows that the MMS conjecture implies Ruckle's conjecture in the case $4k \leq n$, as there is an optimal trivial allocation of gold for every $d$.

\section*{Acknowledgement}
The research was carried out in the REU 2021 programme at E\"otv\"os Lor\'and University in Budapest. We are tremendously grateful to Dömötör P\'alvölgyi and D\'aniel Lenger for supervising us and organising the programme.

\appendix

\section{Hiding strategies in the (3,3,2)-game}\label{Appendix}

In the next table, we present the allocation strategies of the hider. This was also found by a numerical minimizer program, which was running on the long expression describing the winning probability of the searcher based on the hider's parameters. \\
$l_1$: The probability of the hider hiding all treasures in the same box.\\
$l_2$: The probability of the hider hiding each of the treasures in separate boxes.\\
$l_3$: The probability of the hider hiding $2$ treasures in the same box, and the third one in a different box.\\
$q_1$: If the first box contains $2$ treasures, the third one contains $1$, then if in the second question, the searcher asks boxes $1$ and $3$, the hider will give the treasure from box $1$ with $q_1$ probability.\\
$q_2$: If the first box contains $1$ treasure, and the second one contains $2$ treasures, the hider will give the treasure from the first box with $q_2$ probability. \\
$q_3$: If the first box contains 2 treasures, the second contains 1 treasure, and in the first step the searcher received a treasure from the first box, and she asks boxes $1$ and $2$ again, the hider will give her a treasure from box $1$ with $q_3$ probability.\\
$q_4$: If all $3$ boxes contain $1$ treasure, the searcher received a treasure from the first box in the first step, and in the second question she asks boxes $2$ and $3$, the hider will give her the treasure from box $2$ with $q_4$ probability.\\

$
\begin{tabular}{|c|c|c|}
    \hline
      & $\text{Random revealer hiding strategy}$& $\text{Adversary revealer hiding strategy}$  \\ \hline
     $\text{Value of game}$& $12/19$ & $6/10$ \\ \hline
     $ l_1 $ & $5/19$ & $3/10$ \\ \hline
     $ l_2 $ & $2/19$ & $1/10$ \\ \hline
     $ l_3 $ & $12/19$ & $6/10$ \\ \hline
     $ q_1 $ & $1/2 \text{ (fixed)}$ & $1/2$ \\ \hline
     $ q_2 $ & $1/3 \text{ (fixed)}$ & $1/2$ \\ \hline
     $ q_3 $ & $1/2 \text{ (fixed)}$ & $\textcolor{black}{arbitrary}$ \\ \hline 
     $ q_4 $ & $1/2 \text{ (fixed)}$ & $\textcolor{black}{arbitrary}$ \\ \hline
\end{tabular}
$
\vspace{5 mm}

In the following, we demonstrate our method for upper bounds by proving that this strategy of the hider in the random case indeed guarantees a game value of $\frac{12}{19}$.

\subsection{Random revealer}

To calculate the winning probability of the hiding strategy, we need to find the strongest searcher strategy against it. We can assume that the searcher will play according to that. Now notice that whenever the searcher has a choice of what to ask, her best play is asking the one which entails the largest probability of winning out of the options. Hence the strongest search strategy is a definite one, so the treasure discoveries determine exactly what the next query is at each step. Call the value of this hiding strategy $v$. Let $v_1$ be the value of the game if after the first question and receiving a treasure from box $1$, the searcher asks $(1,1,0)$ and then plays optimally, $v_2$ with asking $(1,0,1)$ as a second query, and $v_3$ with asking $(0,1,1)$ as a second query. Then $v=max(v_1,v_2,v_3)$. We can split these subcases further and find the one question that gives the largest probability of winning for the searcher in each subcase.
First question is $(1,1,0)$ with discovered treasure from box $1$.
\begin{itemize}
    \item Second question is $(1,1,0)$, treasure from box $1$ \begin{itemize}
        \item Third question is $(1,1,0)$, probability of winning via that query path: $\frac{5}{19}\frac23=\frac{10}{57}$ from the 3-0-0 starting allocation, $\frac23\frac12\frac4{19}=\frac4{57}$ from the 2-1-0 starting allocation (The $\frac23\frac12$ factor comes from the random revealer's choice in the second step.) This sums to $\frac{14}{57}$.
        \item \textbf{THE STRONGEST} Third question is $(1,0,1)$, probability of winning via that query path: $\frac{5}{19}\frac23=\frac{10}{57}$ from the 3-0-0 starting allocation, $\frac{4}{19}$ from the 2-0-1 starting allocation. This sums to $\frac{22}{57}$.
        \item Third question is $(0,1,1)$, probability of winning via that query path: $\frac23\frac12\frac4{19}=\frac4{57}$ from the 2-1-0 starting allocation, $\frac4{19}$ from the 2-0-1 starting allocation. This sums to $\frac{16}{57}$.
    \end{itemize}
    \item Second question is $(1,1,0)$, treasure from box $2$
    \begin{itemize}
        \item Third question is $(1,1,0)$, probability of winning via that query path: $(\frac{1}{3}+\frac{2}{3}\frac12)\frac4{19}=\frac{8}{57}$ from the 2-1-0 starting allocation. This sums to $\frac8{57}$
        \item \textbf{THE STRONGEST} Third question is $(1,0,1)$, probability of winning via that query path: $\frac23\frac12\frac{4}{19}=\frac{4}{57}$ from the 2-1-0 starting allocation. $\frac{2}{19}$ from the 1-1-1 starting allocation. This sums to $\frac{10}{57}$.
        \item Third question is $(0,1,1)$, probability of winning via that query path: $\frac13\frac{4}{19}=\frac{4}{57}$ from the 1-2-0 starting allocation. $\frac{2}{19}$ from the 1-1-1 starting allocation. This sums to $\frac{10}{57}$.
        \end{itemize}
\item Second question is $(1,0,1)$, treasure from box $1$ \begin{itemize}
        \item \textbf{THE STRONGEST} Third question is $(1,1,0)$, probability of winning via that query path: $\frac{5}{19}\frac23=\frac{10}{57}$ from the 3-0-0 starting allocation, $\frac23\frac4{19}=\frac{8}{57}$ from the 2-1-0 starting allocation. This sums to $\frac{6}{19}$.
        \item  Third question is $(1,0,1)$, probability of winning via that query path: $\frac{5}{19}\frac23=\frac{10}{57}$ from the 3-0-0 starting allocation, $\frac12\frac{4}{19}=\frac{2}{19}$ from the 2-0-1 starting allocation. This sums to $\frac{16}{57}$.
        \item Third question is $(0,1,1)$, probability of winning via that query path: $\frac23\frac4{19}=\frac8{57}$ from the 2-1-0 starting allocation, $\frac12\frac4{19}=\frac2{19}$ from the 2-0-1 starting allocation. This sums to $\frac{14}{57}$.
    \end{itemize}
    \item Second question is $(1,0,1)$, treasure from box $3$ \begin{itemize}
        \item  Third question is $(1,1,0)$, probability of winning via that query path: $\frac12\frac{4}{19}=\frac{2}{19}$ from the 2-0-1 starting allocation, $\frac2{19}$ from the 1-1-1 starting allocation. This sums to $\frac{4}{19}$.
        \item \textbf{THE STRONGEST} Third question is $(1,0,1)$, probability of winning via that query path: $\frac12\frac{4}{19}=\frac{2}{19}$ from the 2-0-1 starting allocation, $\frac{4}{19}$ from the 1-0-2 starting allocation. This sums to $\frac{6}{19}$.
        \item Third question is $(0,1,1)$, probability of winning via that query path: $\frac4{19}$ from the 1-0-2 starting allocation, $\frac2{19}$ from the 1-1-1 starting allocation. This sums to $\frac{6}{19}$.
    \end{itemize}
        \item Second question is $(0,1,1)$, treasure from box $2$ \begin{itemize}
        \item \textbf{THE STRONGEST} Third question is $(1,1,0)$, probability of winning via that query path: $\frac23\frac{4}{19}=\frac{8}{57}$ from the 2-1-0 starting allocation, $\frac13\frac4{19}=\frac{4}{57}$ from the 1-2-0 starting allocation. This sums to $\frac{4}{19}$.
        \item  Third question is $(1,0,1)$, probability of winning via that query path: $\frac23\frac{4}{19}=\frac{2}{19}$ from the 2-1-0 starting allocation, $\frac12\frac{2}{19}=\frac1{19}$ from the 1-1-1 starting allocation. This sums to $\frac{3}{19}$.
        \item Third question is $(0,1,1)$, probability of winning via that query path: $\frac13\frac4{19}=\frac{4}{57}$ from the 1-2-0 starting allocation, $\frac2{19}$ from the 1-1-1 starting allocation. This sums to $\frac{10}{57}$.
        \end{itemize}
            \item Second question is $(0,1,1)$, treasure from box $3$ \begin{itemize}
        \item  Third question is $(1,1,0)$, probability of winning via that query path: $\frac{4}{19}$ from the 2-0-1 starting allocation, $\frac2{19}$ from the 1-1-1 starting allocation. This sums to $\frac{6}{19}$.
        \item \textbf{THE STRONGEST} Third question is $(1,0,1)$, probability of winning via that query path: $\frac{4}{19}$ from the 2-0-1 starting allocation, $\frac{4}{19}$ from the 1-0-2 starting allocation. This sums to $\frac{8}{19}$.
        \item Third question is $(0,1,1)$, probability of winning via that query path: $\frac4{19}$ from the 1-0-2 starting allocation, $\frac2{19}$ from the 1-1-1 starting allocation. This sums to $\frac{6}{19}$.
        \end{itemize}
\end{itemize}

$v_1$ is the sum of the strongest options from the first two cases, therefore $v_1=\frac{22}{57}+\frac{10}{57}=\frac{32}{57}$.
$v_2$ is the sum of the strongest options from the third and fourth cases, therefore $v_2=\frac6{19}+\frac{6}{19}=\frac{12}{19}$.
$v_3$ is the sum of the strongest options from the fifth and sixth cases, therefore $v_3=\frac8{19}+\frac{4}{19}=\frac{12}{19}$.

Since the value $v$ is the maximum of these different strategy values, we can say $v=\frac{12}{19}$ against this random revealer hiding strategy.
\subsection{Adversary revealer}

In a similar manner, we can show that the adversary hiding strategy shows the upper bound of $\frac{3}{5}$ for the adversary revealer game. Again, one of the best searcher strategies against this type of hiding is a definite one by the same argument as in the random case. So we conclude by a case analysis:

\begin{itemize}
    \item Second question is $(1,1,0)$, treasure from box $1$ \begin{itemize}
        \item Third question is $(1,1,0)$, probability of winning via that query path: $\frac{3}{10}\frac23=\frac{2}{10}$ from the 3-0-0 starting allocation, $(1-q_2)q_3\frac2{10}$ from the 2-1-0 starting allocation (The $(1-q_2)q_3$ factor comes from the  revealer's choice in the first and second step.) This sums to $\frac{2+q_3}{10}$.
        \item \textbf{THE STRONGEST} Third question is $(1,0,1)$, probability of winning via that query path: $\frac{3}{10}\frac23=\frac{2}{10}$ from the 3-0-0 starting allocation, $\frac{2}{10}$ from the 2-0-1 starting allocation. This sums to $\frac{4}{10}$.
        \item Third question is $(0,1,1)$, probability of winning via that query path: $(1-q_2)q_3\frac2{10}=\frac{q_3}{10}$ from the 2-1-0 starting allocation, $\frac2{10}$ from the 2-0-1 starting allocation. This sums to $\frac{2+q_3}{10}$.
    \end{itemize}
    \item Second question is $(1,1,0)$, treasure from box $2$
    \begin{itemize}
        \item Third question is $(1,1,0)$, probability of winning via that query path: $(1-q_2)(1-q_3)\frac{2}{10}$ from the 2-1-0 starting allocation. $q_2\frac{2}{10}$ from the 1-2-0 starting allocation. This sums to $\frac{2-q_3}{10}$
        \item Third question is $(1,0,1)$, probability of winning via that query path: $(1-q_2)(1-q_3)\frac{2}{10}$ from the 2-1-0 starting allocation. $\frac{1}{10}$ from the 1-1-1 starting allocation. This sums to $\frac{2-q_3}{10}$.
        \item \textbf{THE STRONGEST} Third question is $(0,1,1)$, probability of winning via that query path: $q_2\frac{2}{10}$ from the 1-2-0 starting allocation. $\frac{1}{10}$ from the 1-1-1 starting allocation. This sums to $\frac{2}{10}$.
        \end{itemize}
\item Second question is $(1,0,1)$, treasure from box $1$ \begin{itemize}
        \item \textbf{THE STRONGEST} Third question is $(1,1,0)$, probability of winning via that query path: $\frac{3}{10}\frac23=\frac{2}{10}$ from the 3-0-0 starting allocation, $(1-q_2)\frac2{10}$ from the 2-1-0 starting allocation. This sums to $\frac{3}{10}$.
        \item  Third question is $(1,0,1)$, probability of winning via that query path: $\frac{3}{10}\frac23=\frac{2}{10}$ from the 3-0-0 starting allocation, $q_1\frac{2}{10}$ from the 2-0-1 starting allocation. This sums to $\frac{3}{10}$.
        \item Third question is $(0,1,1)$, probability of winning via that query path: $(1-q_2)\frac2{10}$ from the 2-1-0 starting allocation, $q_1\frac2{10}$ from the 2-0-1 starting allocation. This sums to $\frac{2}{10}$.
    \end{itemize}
    \item Second question is $(1,0,1)$, treasure from box $3$ \begin{itemize}
        \item  Third question is $(1,1,0)$, probability of winning via that query path: $(1-q_1)\frac{2}{10}$ from the 2-0-1 starting allocation, $\frac1{10}$ from the 1-1-1 starting allocation. This sums to $\frac{2}{10}$.
        \item \textbf{THE STRONGEST} Third question is $(1,0,1)$, probability of winning via that query path: $(1-q_1)\frac{2}{10}$ from the 2-0-1 starting allocation, $\frac{2}{10}$ from the 1-0-2 starting allocation. This sums to $\frac{3}{10}$.
        \item Third question is $(0,1,1)$, probability of winning via that query path: $\frac2{10}$ from the 1-0-2 starting allocation, $\frac1{10}$ from the 1-1-1 starting allocation. This sums to $\frac{3}{10}$.
    \end{itemize}
        \item Second question is $(0,1,1)$, treasure from box $2$ \begin{itemize}
        \item Third question is $(1,1,0)$, probability of winning via that query path: $(1-q_2)\frac{2}{10}$ from the 2-1-0 starting allocation, $q_2\frac2{10}$ from the 1-2-0 starting allocation. This sums to $\frac{2}{10}$.
        \item \textbf{THE STRONGEST} Third question is $(1,0,1)$, probability of winning via that query path: $(1-q_2)\frac{2}{10}$ from the 2-1-0 starting allocation, $q_4\frac{1}{10}$ from the 1-1-1 starting allocation. This sums to $\frac{1+q_4}{10}$.
        \item Third question is $(0,1,1)$, probability of winning via that query path: $q_2\frac2{10}$ from the 1-2-0 starting allocation, $q_4\frac1{10}$ from the 1-1-1 starting allocation. This sums to $\frac{1+q_4}{10}$.
        \end{itemize}
            \item Second question is $(0,1,1)$, treasure from box $3$ \begin{itemize}
        \item  Third question is $(1,1,0)$, probability of winning via that query path: $\frac{2}{10}$ from the 2-0-1 starting allocation, $(1-q_4)\frac1{10}$ from the 1-1-1 starting allocation. This sums to $\frac{3-q_4}{10}$.
        \item \textbf{THE STRONGEST} Third question is $(1,0,1)$, probability of winning via that query path: $\frac{2}{10}$ from the 2-0-1 starting allocation, $\frac{2}{10}$ from the 1-0-2 starting allocation. This sums to $\frac{4}{10}$.
        \item Third question is $(0,1,1)$, probability of winning via that query path: $\frac2{10}$ from the 1-0-2 starting allocation, $(1-q_4)\frac1{10}$ from the 1-1-1 starting allocation. This sums to $\frac{3-q_4}{10}$.
        \end{itemize}
\end{itemize}

$v_1$ is the sum of the strongest options from the first two cases, therefore $v_1=\frac{4}{10}+\frac{2}{10}=\frac{3}{5}$.
$v_2$ is the sum of the strongest options from the third and fourth cases, therefore $v_2=\frac3{10}+\frac{3}{10}=\frac{3}{5}$.
$v_3$ is the sum of the strongest options from the fifth and sixth cases, therefore $v_3=\frac{1+q_4}{10}+\frac{4}{10}=\frac{5+q_4}{10}$.

The value $v$ is the maximum of the above three values, hence it is $v=\frac{3}{5}$ for this adversary hiding strategy, matching the lower bound from the search strategy. Note that the hider can pick $q_3$ and $q_4$ to be an arbitrary number between $0$ and $1$.

\end{document}